\documentclass{amsart}
 \usepackage{amscd}
\usepackage{amssymb}

 \numberwithin{equation}{section}

\newtheorem{theorem}{Theorem}[section]
\newtheorem{lemma}[theorem]{Lemma}
\newtheorem{proposition}[theorem]{Proposition}

\newtheorem{example}[theorem]{Example}

\numberwithin{equation}{section}

\font\math=msbm10
\def\u#1{\hbox{\math#1}}
\def\vZ{{\u Z}}

\def\nin{\noindent}

\def\a{{\alpha}}
\def\b{{\beta}}
\def\g{{\gamma}}

\def\d{{\delta}}

\def\m{{\mu}}

\def\f{{\phi}}

\def\s{{\sigma}}
\def\t{{\tau}}
\def\w{{\omega}}

\def\AA{{\mathcal A}}
\def\BB{{\mathcal B}}
\def\CC{{\mathcal C}}
\def\DD{{\mathcal D}}

\def\FF{{\mathcal F}}
\def\GG{{\mathcal G}}
\def\HH{{\mathcal H}}

\def\al{{\aleph}}

\def\fg{finitely generated}

\def\fr{finite rank}
\def\tf{torsion-free}

\def\rk{\mathop{\rm rk}}

\def\fg{finitely generated}

\def\cd{completely decomposable}
\def\fd{finitely decomposable}

\def\cwac{continuous well-ordered ascending chain }
\def\fr{finite rank}
\def\tf{torsion-free}

\def\vd{valuation domain}

\def\Dds{Dedekind domains}
\def\Pd{Pr\"ufer domain}
\def\Pds{Pr\"ufer domains}

\def\<{\langle}
\def\>{\rangle}

\long\def\alert#1{\smallskip{\hskip\parindent\vrule%
\vbox{\advance\hsize-2\parindent\hrule\smallskip\parindent.4\parindent%
\narrower\noindent#1\smallskip\hrule}\vrule\hfill}\smallskip}

\begin{document}

\title{On Completely Decomposable and Separable Modules over Pr\"ufer Domains}

\author{L\'aszl\'o Fuchs}
\address{Department of Mathematics, Tulane University, New Orleans, Louisiana 70118, USA}
\email{fuchs@tulane.edu}

\author{Jorge E. Mac\'{\i}as-D\'{\i}az}
\address{
Departamento de Matem\'aticas y Fisica, Universidad Aut\'onoma de Aguascalientes, Blvd. Universidad 940, Ciudad 
Universitaria, Aguascalientes, Ags. 20100, Mexico}
\email{jemacias@correo.uaa.mx}



\maketitle

 \begin{abstract}   We generalize  known results on summands of \cd\  and separable \tf\  abelian groups to modules over $h$-local \Pd s.  Over such domains summands of \cd\ \tf\ modules  are again \cd\  (Theorem \ref{cd}) and summands of   separable \tf\  modules are likewise separable (Theorem \ref{main}).  In addition, a Pontryagin-Hill type theorem is established on countable chains of homogeneous \cd\ modules over $h$-local \Pd s (Theorem \ref{mainth}).
 
 Several auxiliary results are proved for modules over integral domains that are direct sums of finite or countable rank submodules.
 \end{abstract}

\section {Introduction}  \medskip

All modules in this note are \tf\ modules over integral domains $R$. 

 By a {\it \cd} \tf\ module $M$ is meant a direct sum of rank 1 modules, i.e. of modules that are $R$-isomorphic to submodules of the field $Q$ of quotients of $R$.  The cardinal number of the set of summands is called the {\it rank} of $M$, in notation: $\rk M$. This is an invariant of $M$:  the cardinality of every maximal independent set in $M$.

By making use of results by Olberding \cite{O},  recently Goeters \cite{G} proved that over an $h$-local \Pd\ $R$ summands of finite rank \cd\ \tf\ modules are again \cd.  In Theorem \ref{cd} we extend this theorem to modules
of arbitrary ranks.  Our approach is different from Goeters inasmuch as we rely on results by Kolettis \cite{Ko} on homogeneously decomposable \tf\ modules. Our theorem generalizes the celebrated Baer-Kulikov-Kaplansky theorem on summands of \cd\ abelian groups (e.g. Fuchs \cite[Theorem 86.7]{F}).

We also generalize  an old  result on abelian groups stating that summands of separable \tf\ groups are again separable (see e.g. Fuchs \cite[Theorem 87.5]{F}).  Theorem \ref{main} asserts that summands of  separable \tf\ modules over an $h$-local \Pd\ $R$ are again separable.  The proof is {\sl via} reduction to the \cd\ case.

Hill \cite{H} established  a far-reaching generalization of Pontryagin's criterion on the freeness of abelian groups by proving that the union of a countable ascending chain of pure free subgroups (of any size) is likewise free. This theorem is extended here to countable chains of homogeneous \cd\ modules over  $h$-local \Pd s  (Theorem \ref{mainth}).  The hypothesis of purity had to be  strengthened:  we assume that countable rank RD-submodules in the factors of the chain  can be obtained as  images of countable rank submodules from the links of the chain (they are called   RD$^*$-submodules) | a condition that is automatically satisfied whenever $R$ is a countable domain.

We also show that there is a \cwac\  with countable rank factors consisting of \cd\ RD-submodules between a \cd\  module and a \cd\ RD$^*$-submodule  | a fact that underlines the importance played by countability in the theory of \cd\ modules.  (This phenomenon was first observed by Dugas-Rangaswamy \cite{DR} for abelian groups.)

Some of our results are proved under more general conditions than needed for our main results: for direct sums of  finite or countable rank modules (rather than just for direct sums of rank 1 modules).  Besides their independent interest, their proofs also reveal  basic ideas on which the results rely.

We wish to thank the referee for his/her valuable comments.

%=======================================

\section {Preliminaries}  \medskip

Let $M$ be  any  module over the domain $R$. Following P. Hill, we define various families of submodules (see also Fuchs-Salce \cite{FS}). 

By an  {\it $H(\al_0)$-family} in $M$ is meant a
collection $\HH$  of submodules of $M$ satisfying the following properties:

H1. $0, M \in \HH$; 

H2. $\HH$ is closed under unions, i.e. $M_i \in \HH \ (i \in I)$
implies 
$\sum_{i \in I} M_i \in \HH$ for any index set $I$;

H3.  if $C \in \HH$ and $X$ is any countable subset of $M$,
 then there is a submodule $B \in \HH$ that contains both $C$ and
$X$  and is such that $B/C$ is countably generated. 
 
 A {\it $G(\al_0)$-family} $\GG$ is defined similarly with H2
replaced by the following weaker condition:

G2. $\GG$ is closed under unions of chains.
 
 In this paper we are interested in the rank versions of these families. The
{\it $H^*(\al_0)$-family} and the {\it $G^*(\al_0)$-family} are defined similarly for \tf\ modules $M$ (see Rangaswamy \cite{R}): in these cases  the submodules in the families are required to be RD-submodules and in condition H3  `countable rank' is to be used in place  of `countably generated.'   (Recall: a relatively divisible or briefly an RD-submodule  of $M$ is a submodule $N$ satisfying $rN = N \cap rM$ for all $r \in R$.)

Obviously, every $H^*(\al_0)$-family is a $G^*(\al_0)$-family, but
the converse fails in general.   Note that every \tf\ $R$-module $M$ has a  $G^*(\al_0)$-family of RD-submodules. 
  In fact, select a maximal independent set $X$ in $M$. For a subset $Y$ of $X$, let $M_Y$ denote the smallest RD-submodule of $M$ that contains $Y$. It is readily checked that the set of all $M_Y$ is a $G^*(\al_0)$-family. However, this is in general not an $H^*(\al_0)$-family, since the sum of  two RD-submodules need not be an RD-submodule.

 If the $R$-module $M$ is a direct sum of submodules of countable rank, and $M= \oplus_{\a \in I} A_\a$ with $\rk A_\a \le \al_0$ is such a decomposition for an index set $I$, then the standard way of defining an $H^*(\al_0)$-family $\HH$ of summands in $M$ is to consider the set of all partial summands in this decomposition: $H_J= \oplus_{\a \in J} A_\a$ with $J$ ranging over all subsets of $I$. 

It is well known (and is easy to check) that the intersection of a finite number of (or of even countably many)  $G^*(\al_0)$-families is again such a family.  The same holds for $H^*(\al_0)$-families.
\medskip

Next we introduce a new concept that  will be needed in the sequel, strengthening the RD-property of submodules.

Let $A$ be a submodule of the \tf\ $R$-module $M$, and $\f: M \to M/A$ the canonical map.  We say that $A$ is an {\it strong RD$^*$-submodule} of $M$ if 

1) it is an RD-submodule, and 

2) each  finite (and hence countable) rank submodule in $M/A$ has a countable rank preimage in $M$. 

For the sake of comparison  let us point out  that the RD-submodule $A$ is   {\it balanced} in $M$ if  every  rank one submodule in $M/A$ has a  rank one preimage in $M$. Thus the property of being `RD*' lies between `RD' and `balancedness'.

In the following list (a)-(d),  $A , B$ will denote RD-submodules of the \tf\ $R$-module $M$ such that  $A \le B$. It is straightforward to verify that 

(a) direct summands and balanced submodules are RD*-submodules;

(b)  if $A$ is an RD*-submodule of $M$, then it is RD* in $B$ as well;  

(c) the property 'RD*'  is a transitive relation: if $A$ is an RD*-submodule of $B$ and  $B$ is an RD*-submodule of $M$, then $A$ is an RD*-submodule of $M$;  

(d) let $A$ be an RD*-submodule of $M$;  then $B$ is an RD*-submodule in $M$ if and only if $B/A$ is an RD*-submodule of $M/A$.  

   \begin{example} \rm Suppose that there exists an uncountably generated rank one \tf\ $R$-module $A$ (e.g. an uncountably generated field of quotients of certain \Dds).  If $0 \to H \to F \to A \to 0$ is a free presentation of $A$, then $H$ is RD, but not RD* in $F$.
  \end{example}

\begin{example} \rm It is easy to see that the concept of RD$^*$-submodule is new only if $R$ is uncountable, because if $R$ is a countable domain, then all RD-submodules are automatically RD*-submodules.  In fact,  if $A$ is RD in $M$ and $\f: M \to M/A$ is the canonical map, then every countable rank submodule of $M/A$ is countably generated, and the generators are included in $\f C$ for some countable rank submodule $C$ of $M$. 
  \end{example} 
  
  Next we prove an easy result. 
 
 \begin{lemma} \label{RD*} If $A$ is an RD$^*$-submodule of the \tf\  $R$-module $M$, then for every $G^*(\al_0)$-family $\CC$ of RD-submodules in $M/A$,  $M$ admits a $G^*(\al_0)$-family $\GG$ of RD-submodules such that $$\CC= \{ \f B \ | \ B \in \GG\},$$ 
 where $\f$ denotes the canonical projection $M \to M/A$. \end{lemma}
 
 \begin{proof} Let $\FF$ be the $G^*(\al_0)$-family of RD-submodules of $M$ and $\CC$  a $G^*(\al_0)$-family of RD-submodules in $M/A$,  Define $\GG = \{B \in \FF \ | \ \f B \in \CC\}$ where $\f: M \to M/A$ denotes the canonical projection. 
 It is readily seen that  $\GG$ is as desired.
 \end{proof}

We say that two \tf\  $R$-modules, $A$ and $B$, are {\it quasi-isomorphic} (see Goeters \cite{G}) if there exist submodules $A' \le A$ and $B' \le B$ such that $A' \cong B$ and $B' \cong A$. Quasi-isomorphism is evidently an equivalence relation on \tf\ $R$-modules. 

The equivalence classes of rank 1 \tf\ $R$-modules under quasi-isomorphism are called {\it types}.  The type of a rank 1 \tf\ module $M$ is denoted by the symbol $\t(M)$. The set of types admits a natural partial order: for types $\s$ and $\t$ we set $\s \le \t$ if and only if there exist rank 1 $R$-modules  $A  $ and $B  $ with $\t(A)=\s$ and $\t(B)=\t$ such that $A$ is a submodule of $B$. The smallest type is the common type of all fractional ideals of $R$, while the largest type is the type of $Q$, the field of quotients of $R$.

Just as for abelian groups, with a given type $\t$ one can associate two fully invariant submodules, $M(\t)$ and $M^*(\t)$, of a \tf\ $R$-module $M$ as follows:
$$M(\t)= \sum \{X \ | \ X \le M; \t(X) \ge \t\}$$ 
and 
$$M^*(\t)= \sum \{X \ | \ X \le M; \t(X) > \t\} $$ 
where $X$ stands for rank one submodules.  From the definition it is clear that they are submodules of $M$ such that $M(\t) \ge M^*(\t)$; furthermore, $M(\s) \le M(\t)$ and $M^*(\s) \le M^*(\t)$ whenever $\s \ge \t$. 

A \tf\ module $H$ will be called {\it homogeneous of type} $\t$ if $H(\t) =H$ and $H^*(\t)=0$.  Evidently, RD-submodules of homogeneous \tf\ modules are again homogeneous. Projective modules as well as divisible \tf\ modules are homogeneous, and so are direct sums of fractional ideals of $R$. 

Kolettis \cite{Ko} calls a \tf\ module $M$ {\it homogeneously decomposable} if it is a direct sum of homogeneous modules (of equal or different types). He proves that a \tf\ module $M$ of countable rank is homogeneously decomposable if and only if it satisfies the following two conditions: (i) for every type $\t$, both $M(\t)$ and $M^*(\t)$ are summands of $M$; and (ii) every element of $M$ belongs to a direct summand of $M$ that is a finite direct sum of homogeneous modules. Using this characterization, he proves:

\begin{theorem} {\rm (Kolettis \cite{Ko})} \label{K} Summands of a homogeneously decomposable \tf\  $R$-module are themselves homogeneously decomposable.  \qedsymbol
\end{theorem}

\bigskip
%=======================================

\section {Summands of completely decomposable modules} \medskip

We repeat the definition: a \tf\ $R$-module $C$ is   {\it completely decomposable} if it is the direct sum of rank 1 submodules. Such a $C$ is homogeneous if it is the direct sum of quasi-isomorphic rank 1 modules.  It is   clear that completely decomposable modules are homogeneously decomposable.

In the study of completely decomposable modules it is crucial what happens in the finite rank case.  It is a classical theorem by R. Baer \cite{B} that a finite rank completely decomposable homogeneous abelian group  has the distinguished property that every pure (i.e. RD-)subgroup is a summand, and hence it is likewise completely decomposable. This is not true in general, not even for projective modules. Olberding \cite{O} proved  that this property  is shared by $h$-local \Pd s $R$  (recall: a domain $R$ is $h$-{\it local} if every non-zero element belongs only to finitely many maximal ideals, and every non-zero prime ideal is contained only in a single maximal ideal), moreover:

\begin{theorem} {\rm (Olberding \cite{O})} \label{finite} The following are equivalent for any integral domain $R$:

{\rm (a)} $R$ is an $h$-local \Pd;

{\rm (b)} every pure submodule of a finite rank completely decomposable homogeneous \tf\ module is a summand;

{\rm (c)} every pure submodule of a finite direct sum of fractional ideals is a summand.  
 \qedsymbol
\end{theorem}

It is easy to see that in conditions (b) and (c) `pure submodule' can be replaced by `RD-submodule' (this strengthens the hypothesis of the difficult implication (c) $\Rightarrow$ (a)). It also follows at once that the summands in (b) and (c) are then completely decomposable. 

Using Olberding's theorem,  Goeters \cite{G}  proved that summands of finite rank completely decomposable \tf\ modules over $h$-local \Pd s are again completely decomposable.  Our present goal is to extend this result to completely decomposable modules of arbitrarily high ranks and to verify the analogue for separable modules (see next section).   We call a \tf\ $R$-module $M$   {\it separable} (in the sense of  Baer \cite{B})  if  1) every finite set of its elements can be embedded in a finite rank summand of $M$, and 2) \fr\ summands of $M$ are \cd. (This is a slightly stronger definition than  the one used in Fuchs-Salce \cite[Chapter XVI, section 5]{FS}.) 

Accordingly, we are now going to prove:

\begin{theorem} \label{cd} Summands of completely decomposable \tf\ modules over $h$-local \Pd s are likewise completely decomposable.
\end{theorem} 

\begin{proof} The proof begins with the reduction to the countable rank case. By the rank version of a well-known theorem by Kaplansky \cite{K}, summands of modules that are direct sums of countable rank submodules are themselves direct sums of countable rank summands.  In view of this, it is straightforward to see that it will suffice to prove that if $M = A \oplus B$ is a {\sl countable rank} completely decomposable \tf\ $R$-module, then $A$ is also completely decomposable. 

Further reduction is possible if we make use of Kolettis' theorem quoted above. Indeed, a completely decomposable module being homogeneously decomposable, from Theorem \ref{K} it follows that for the proof of Theorem \ref{cd} we may assume without loss of generality that $M$ is {\sl homogeneous}.

The next step in the proof is to show that the summand $A$ of $M$ is  separable. So, let $a_1, \dots, a_n$ be elements of $A$.  Clearly, there is a finite rank  \cd\ summand $N$ of $M$ that contains all of $a_1, \dots, a_n$. The RD-submodule $A'$ spanned by the elements $a_1, \dots, a_n$ is by Olberding's theorem a  \cd\ summand of $N$. Thus $A'$ is a  completely decomposable summand of $M$, and hence of $A$. This shows that all \fr\ RD-submodules are \cd\ summands, establishing the separability of $A$.  

Thus $A$ is the union of a countable chain of \fr\ \cd\ submodules each of which is a summand of the following ones with \cd\  complements.  It follows    that $A$ is  completely decomposable, completing the proof of the theorem.  
\end{proof} 

\bigskip
%=======================================

\section {Summands of separable modules} \medskip

   We start the discussion of separability (defined above) with a general lemma that holds over all integral domains.

  \begin{lemma} \label{sep} A domain $R$ has the property that summands of separable torsion-free $R$-modules are again separable if and only if summands of completely decomposable torsion-free $R$-modules of countable rank  are again  \cd.  
\end{lemma}

\begin{proof} Before proving the equivalence of the stated conditions, we observe that either implies that summands of \fr\ \cd\ $R$-modules are again \cd. As a consequence, we can argue (as in the final part of the proof of Theorem \ref{cd}) that countable rank separable $R$-modules are \cd.

Necessity follows at once by applying the hypothesis to a \cd\ module of countable rank noting that countable rank separable modules are  \cd.

 For sufficiency, assume that summands of completely decomposable torsion-free $R$-modules of countable rank  are    \cd. Let $M$ be a separable \tf\  $R$-module, and $M = A \oplus B$ a direct decomposition of $M$. Given a finite subset $S$ in $A$, we have to show that $S$ is contained in a \fr\ summand $H$ of $A$ and \fr\ summands of $A$ are \cd.

Let $M_0$ be a \fr\ \cd\ summand of $M$ containing $S$, and let $A_0, B_0$ be \fr\ RD-submodules of $A$ and $B$, respectively, such that $M_0 \le A_0 \oplus B_0$.  There is a \fr\ \cd\ summand $M_1$ of $M$ that contains a maximal independent set in $A_0 \oplus B_0$, and hence it contains both $A_0$ and $B_0$.  Furthermore,  there are  \fr\ RD-submodules $A_1, B_1$ of $A$ and $B$, respectively, satisfying $M_1 \le A_1 \oplus B_1$.  Continuing this way, we obtain an ascending chain
$$M_0 \le A_0 \oplus B_0\le M_1 \le A_1 \oplus B_1 \le \dots \le M_n \le A_n \oplus B_n  \le  \dots   \quad (n < \w)$$
where $M_n$ are finite rank summands of $M,$ while $ A_n, B_n$ are finite rank RD-submodules of  $A, B$. The union $M'$ of this chain is a countable rank submodule of $M$ which is \cd\ as the union of the chain of \cd\ modules $M_n$ where every module in the chain is a summand in each of the following ones with \cd\ complement.  Moreover, by construction, we have 
$$M' = A' \oplus B' \quad {\rm where}\ \ A'= \bigcup_n A_n, \ B'= \bigcup_n B_n.$$ 
By hypothesis, $A', B'$ are \cd\ as summands of the \cd\ module $M'$ of countable rank. Therefore, $S$ is contained in a \fr\ \cd\ summand $H$ of $A'$.  Then $H$ is a summand of $M'$, and since $H \le M_k < M'$ for some $k < \w$, $H$ is a summand of  $M_k$,  so also of $M$, and hence of $A$.

From our argument it is also clear that \fr\ summands of $A$ are summands of a \cd\ module, so they are themselves \cd. 
\end{proof}

Consequently, combining Theorem \ref{cd} and Lemma \ref{sep} we can state:

\begin{theorem} \label{main} Summands of  separable torsion-free modules over an $h$-local \Pd\ are  separable.   \qedsymbol \end{theorem}

%=======================================

\bigskip

\section {Chains of \cd\  submodules between \cd\ submodules} \medskip
 
  We would like to call attention to an interesting phenomenon:  the existence of chains with countable rank factors between a  \cd\ module and  a  \cd\ RD-submodule; see Proposition \ref{chainbetween}.  This has been pointed out for abelian groups by Dugas-Rangaswamy \cite{DR} (cf. also Fuchs-Viljoen \cite{FV}), and interestingly, it holds over  arbitrary integral domains. It provides an additional evidence that complete decomposability is intimately tied to countability even in more general situations.
  
 We phrase the results more generally, for modules that are direct sums of   countable rank submodules.    
 The \cd\ case will then be a simple corollary. 
 
 We require a preliminary lemma. 
   
 \begin{lemma} \label{subcd} Suppose $ B $ is an $R$-module  that is a direct sum of modules of countable rank,  and  $A$ is a submodule of $B$ that is likewise a direct sum of countable rank modules.
  
  {\rm (i)}  If $B'$ is a summand of $B$ such that  $A'= A \cap B'$ is a summand of $A$, then $A +B'$ is  a direct sum of modules of countable rank.
  
   {\rm (ii)} There exist $\GG^*(\al_0)$-families $\AA$ and $ \BB$ of summands in $A$ and $B$, respectively, such that $\AA = \{A \cap X \ | \ X \in \BB\}$. 
 \end{lemma}
  
    \begin{proof}  (i)  By a well-known Kaplansky result \cite{K} already mentioned above, summands of a module that is a direct sum of modules of countable rank are again direct sums of modules of countable rank. Consequently, $(A+B')/B' \cong A/A'$ is a module that is a direct sum of modules of countable rank.  Furthermore, $B'$ is a summand of $A+B'$, thus $A+B' \cong B' \oplus A/A'$ is likewise a direct sum of submodules of countable rank.
    
    (ii)  Fix   decompositions of $A$ and $B$ as direct sums of countable rank modules, and let $\AA'$ and $\BB'$ denote the  $\HH^*(\al_0)$-families  of direct sums of subsets of these summands in $A$ and $B$, respectively,  The first and the second entries in the pairs $(A', B')$ with $A' = A \cap B' \ (A' \in \AA', B' \in \BB')$ yield the desired $\GG^*(\al_0)$-families $\AA$ and $ \BB$, respectively.  \end{proof}

   We can now verify: 

  \begin{proposition} \label{chainbetween} Suppose $ A $ is an RD$^*$-submodule of  the \tf\ $R$-module $B$ such that both $A$ and $B$ are direct sums of countable rank submodules.  
  
    {\rm (i)} For some ordinal $\t$, there is a \cwac\  
     $$ A=B_0 \le B_1 \le \dots \le B_\s \le \dots \le  B_\t=B \leqno (1) $$
of  RD-submodules  between $A$ and $B$ such that  each $B_\s$ is a direct sum of submodules of countable rank and 
$B_{\s+1} /B_\s $ is \tf\  of   rank $\le \al_0$,  for every $\s < \t . $
 
  {\rm (ii)}  If  $A$ and  $B$ are \cd, then the  $ B_\s $ can be chosen to be \cd\  as well.  \end{proposition} 

 \begin{proof}  (i) Select $\GG^*(\al_0)$-families $\AA$ and $ \BB$ of summands in $A$ and $B$, respectively,  as stated in Lemma \ref{subcd} (ii).  In view of Lemma \ref{RD*}, we can find in $B$ a $\GG^*(\al_0)$-family $\GG$ of RD-submodules $B'$ such that $A+B'$ is always an RD-submodule of $B$.  The intersection $\BB \cap \GG$ is evidently a $\GG^*(\al_0)$-family, from which we extract a \cwac\  $ 0= B'_0 < B'_1 <\dots < B'_\s < \dots < B'_\t=B $ such that all $B'_{\s+1} /B'_\s$ are of countable rank.  Next we form a chain (1) with the RD-submodules $B_\s = A+B'_\s \ (\s < \t)$. Lemma \ref{subcd}(i) guarantees that the chain (1)  will have the desired property, since $$B_{\s+1} /B_\s \cong B'_{\s+1} / [B'_{\s+1} \cap (A+B'_{\s})] = B'_{\s+1} / [(B'_{\s+1} \cap A)+B'_{\s}] $$ is a surjective image of $B'_{\s+1} /B'_\s$. 
 
 (ii) In case both $A$ and $B$ are completely decomposable, then  the summands $A', B'$ in  Lemma  \ref{subcd}(i) can be chosen such that all the modules $A/A'$ and $B'$ are \cd.  Then the modules
$B_\s $ of the preceding paragraph will be  \cd.     \end{proof}

\bigskip

%=======================================

  \section {Chains of \fd\  modules} \medskip

The classical Pontryagin theorem on \tf\ abelian groups states that the union of an ascending sequence of finite rank free groups is free whenever each group in the sequence is pure in its immediate successor.  This important theorem has been generalized by Hill \cite{H}: the union of an ascending sequence $0 = A_0 < A_1<  \dots 
< A_n < \dots \quad (n < \w) $ of free abelian groups (of any size) is free provided that for each $n < \w$, $A_n $ is pure in $A_{n+1}$.  Our next goal is to establish an analogous result for homogeneous \cd\ modules over an $h$-local \Pd\ (Theorem \ref{mainth}).  (A similar result on \vd s was proved by Rangaswamy \cite{R}.)  In this section, we prove a preparatory result (Theorem \ref{fd}) that might be of independent interest. 
 It is phrased in more general terms than needed in what follows in order to emphasize a main point that makes things work for countable unions.

 By a {\it finitely decomposable} \tf\ $R$-module we mean a module that is the direct sum of \fr\ submodules. 
 We call an RD-submodule $A$ of the \tf\ $R$-module $M$ {\it ultra-balanced} if $A$ is a summand in every RD-submodule $C$ of $M$ that contains $A$ as finite corank submodule. (Ultra-balanced subgroups of abelian groups have been introduced and discussed by T. Chao \cite{C}. Ultra-balanced submodules are of course balanced.)
 The meaning of {\it `ultra-balanced projective'} is evident. It is straightforward to check that  the  ultra-balanced projective modules are precisely the summands of   \fd\ modules. They are not necessarily \fd, not even for abelian groups; this is demonstrated by an example of Corner \cite{Co}: a countable \fd\ \tf\ abelian group that is the direct sum of two indecomposable groups of countable rank. 

We now state the crucial lemma (some arguments are similar e.g. to \cite[Lemma 5.2]{R}). 

\begin{lemma} \label{crucial}   Assume that  the  $R$-module $M$ is the union of an ascending chain
   $$0 = M_0 < M_1<  \dots < M_n < \dots \quad (n < \w) \leqno (2)$$
of \tf\ submodules such that 

{\rm (i)}  each $M_n$ admits a $G^*(\al_0)$-family $\DD_n$ of direct summands, and 
 
 {\rm (ii)}   for each $n < \w$, $M_n$ is  an RD*-submodule in $M_{n+1}$.    

\nin  Then there exists a $G^*(\al_0)$-family $\BB$ of  ultra-balanced submodules of $M$ such that for all $n < \w$ and for all $A \in \BB$ we have

{\rm (a)} $A \cap M_n \in \DD_n$; and 

{\rm (b)} $ A + M_n$ is an RD-submodule in $M$.
 \end{lemma}
 
\begin{proof}  Assume (2) satisfies hypotheses  (i) and (ii). First of all, we claim that  the collection
$$\BB_n = \{ A \in \DD_n \ | \ A+ M_k\ {\rm  is\  RD} \ {\rm in} \ M_n   \ (k < n)\}$$ 
is a $G^*(\al_0)$-family of summands in $M_n$. 
 By hypothesis (ii), $M_n$ has  a $G^*(\al_0)$-family $\GG_k \ (k< n)$ of RD-submodules such that its members project onto RD-submodules of $M_n/M_k$ (see Lemma \ref{RD*}).  It is readily checked that  $$\BB_n= \DD_n \cap  \GG_1 \cap \dots \cap \GG_{n-1}$$ is as desired.

 The next step is to show that   the collection
$$\BB = \{A \leq M \ | \ A \cap M_n \in \BB_n \  {\rm for\ each} \  n <\w\}$$
is a $G^*(\al_0)$-family of RD-submodules in $M$. 
 For details, we refer to the proof of \cite[Lemma 1.7]{FH}.  
   It follows easily that the $G^*(\al_0)$-family $\BB$ of RD-submodules  will have properties (a) and (b).  E.g. to check condition (b) just observe that the RD-property is transitive and $A = \bigcup_n (A  \cap M_n)$.
 
 It remains to show that the submodules in $\BB$ are ultra-balanced in $M$. Suppose $A \in \BB$,  and let $C$ be an RD-submodule of $M$ such that $A  < C$ with $C/A$ of finite rank. Pick a maximal independent set $S=\{c_1, \dots, c_k\}$ in $C$ mod $A $. There is an index $n$ such that $S \subset M_n$.    By (b), $A + M_n$ is an RD-submodule in $M$, and the same is true for $A + (  M_n  \cap C) =  (A + M_n) \cap C$.   This RD-submodule contains both $A$ and $S$, consequently, $A  + ( M_n  \cap C) =C$. By (a), $M_n \cap A  $ is a summand of $M_n$, say, $M_n =  (M_n \cap A) \oplus B$.  Therefore,  $ M_n \cap C = (M_n \cap A ) \oplus ( B \cap C)$, whence $$C = A+  (M_n \cap A) +(  B \cap C)   = A + ( B \cap C)$$  follows.  Since $A \cap  B \cap C  = A   \cap B  = A  \cap B \cap M_n =0$, we have $C= A  \oplus (  B \cap C)$.  Here $ B \cap C$ is a finite rank RD-submodule of $M$, so $A$ is a summand of every submodule of $M$ in which it is contained with finite corank, i.e. $A$ is ultra-balanced in $M$.
  \end{proof} 

The countable rank version of Theorem \ref{fd} is proved separately as our next lemma.

\begin{lemma} \label{cfd} Assume  $(2) $ is a chain of \tf\ $R$-modules of countable rank such that

 {\rm (a)}  each $M_n$ is \fd; 

{\rm (b)} $M_n$ is an RD-submodule of $M_{n+1}$ for each $n< \w$.

 \nin A necessary and sufficient condition that the union $M$ of the chain be \fd\ is 
 
$ (^*)$ for every finite set $S$ of elements in $M$ there exist an index $n$ and a \fr\   submodule $C$ of $M$ containing $S$ such that $C$ is a summand of $M_m$ for all $m \ge n$.  
\end{lemma}

\begin{proof} Necessity is easy: if $M$ is \fd, then it must have a \fr\ summand $C$ containing a given finite set of elements, and  $C$ is necessarily a summand of each $M_n$ in which it is contained.  

For the proof of sufficiency, assume the stated condition. 
Select a maximal independent set $a_0, a_1, \dots, a_n, \dots$ of $M$. We construct a chain $C_0 \le C_1 \le \dots \le C_n \le \dots$ of submodules satisfying the following conditions:

$(\a)$ $a_0, a_1, \dots, a_n \in C_n$ for each $n < \w$;

$(\b)$ $C_n$ is a \fr\  summand of all $M_m$ for all $m \ge i_n$ for some $i_n$;

$(\g)$ $i_0 \le i_1 \le \dots \le i_n \le \dots$.

\nin Hypothesis $(^*)$ guarantees that such a chain does exist.  Clearly, $C_n$ will be a summand of $C_{n+1}$, because it is a summand of $M_{i_{n+1}}$ containing $C_{n+1}$; say, $C_{n+1} = C_{n} \oplus B_{n+1}$. Then $M$ will be the direct sum of $C_0$ and the $B_n$'s all of which are of \fr. Consequently, $M$ is \fd.
\end{proof}

Observe that the proof of the preceding lemma establishes the necessity of the condition  $ (^*)$  in the following theorem. 

\begin{theorem} \label{fd} Let $(2) $ be a chain of \tf\ $R$-modules.  Suppose that 

 {\rm (a)}  each $M_n$ is \fd; 

{\rm (b)} $M_n$ is an RD$^*$-submodule of $M_{n+1}$ for each $n< \w$.

 \nin A necessary and sufficient condition that the union $M$ of the chain be \fd\ is condition 
 $ (^*)$  in Lemma  \ref{cfd}. \end{theorem}

\begin{proof} Assuming  $ (^*)$,  let $\DD_n$ denote an $H^*(\al_0)$-family of summands in a fixed direct decomposition of $M_n$ as a direct sum of \fr\  submodules. We appeal to Lemma \ref{crucial} to conclude  that there is a $G^*(\al_0)$-family $\BB$ of  ultra-balanced submodules of $M$ such that $A \cap M_n \in \DD_n$ and 
$ A + M_n$ is an RD-submodule in $M$ for every $A \in \BB$ and for every $n < \w$. 

By transfinite induction we construct, for some ordinal $\m$,  a \cwac
$$ 0 = N_0 < N_1 < \dots < N_\a < \dots  \qquad (\a < \m) \leqno (3)$$
of  submodules of $M$ such that, for each $\a < \m$,

{\rm (i)} $N_\a$  is  \fd;

{\rm (ii)} $N_\a \in \BB$; 

{\rm (iii)} $N_\a$  is a summand in $N_{\a+1}$; 

{\rm (iv)}  for a finite subset $S$ of $N_\a$, $N_\a$  has a \fr\   summand $C$ of $M$ that contains $S$  and is a summand of $M_m$ for all $m \ge n$, for a suitable $n$;  

{\rm (v)} $ N_{\a+1}/N_\a$ is \fd\ of rank $ \le \al_0$;

{\rm (vi)} $M = \bigcup_{\a < \m} N_\a$.

\nin It will suffice to discuss the step from $N_\a$ to $N_{\a+1}$ for $\a < \m$.  So suppose that, for some ordinal $\b < \m$, the submodules $N_\a$ have been defined for all $\a \le \b$ satisfying (i)-(v). Pick a countable independent set $a_0, a_1, \dots, a_n, \dots$ modulo $N_\b$ in  $M$,  and proceed to construct a chain $C_0 \le C_1 \le \dots \le C_k \le \dots$ satisfying  conditions $(\a)$-$(\g)$ for the chosen elements $a_n$.  Moreover, in order to satisfy (iv), we require that the $C_k$ are such that

$(\d)$ $C_k \cap M_{i_k} \in \BB_{i_k} $ for each $k < \w$.

\nin This can be achieved if we increase the $C_k$  by including an appropriate \fr\ summand of $N_\b$.  
Then $N_\b \cap C_k= X_k$ will be a summand of $N_\b$, say, $N_\b = X_k \oplus P_k$. Furthermore,   by (ii) $N_\b$ is ultra-balanced in  $N_\b+C_k$, so  $N_\b / P_k \cong X_k$ is ultra-balanced in $(N_\b+C_k) / P_k$ whence $C_k = X_k \oplus Y_k$ follows for a suitable \fr\ submodule $Y_k$ of $M$. Similarly, we obtain $C_{k+1} = X_{k+1}  \oplus Y_{k+1} $.  Manifestly, these $Y_k \ (k < \w)$ form an ascending chain mod $N_\b$, and we set $$N_{\b+1}= \cup_{k < \w} (N_\b \oplus Y_k).$$

In order to verify (v) for index $\b$, we show that $Y_k$ is a summand of $Y_{k+1}$ mod $N_\b$.  We argue as follows. Write $C_{k+1} = C_k \oplus D_k$ for $k < \w$.  As $D_k$ is of \fr, we have $N_\b + D_k = N_\b \oplus V_k$ for some \fr\ module $V_k$ (again by the ultra-balancedness of $N_\b$).   In addition,
$$N_\b \oplus  Y_{k+1} = N_\b + C_{k+1} = N_\b + C_k+ D_k =(N_\b + D_k)+ C_k =$$
$$ =(N_\b \oplus V_k) + X_k + Y_k = N_\b + V_k + Y_k .$$ 
We claim that the last sum is actually a direct sum, and prove this by comparing ranks. If we denote the ranks of $Y_k, Y_{k+1}, V_k$ by $r, s, t$, respectively, then these are  also the ranks of $C_k, C_{k+1}, D_k$ modulo $N_\b$, so from $C_{k+1} = C_k \oplus D_k$ we obtain $s \ge r+t$.  This suffices to conclude that $N_\b + V_k + Y_k = N_\b \oplus V_k \oplus Y_k$,
which implies that $Y_{k+1} \equiv   Y_k \oplus V_k$ mod $N_\b$, a desired.  The proof can be finished by the same argument as in the proof of Lemma \ref{cfd}.
\end{proof}

%=======================================

  \section {A main result} \medskip

We are  now  prepared for the proof of a main result (a somewhat weaker form was included in the Ph.D. thesis of the second author   \cite{M}). It generalizes the Pontryagin-Hill theorem from free abelian groups to homogeneous \cd\ modules over $h$-local \Pds.

\begin{theorem} \label{mainth} Let  $R$ be an $h$-local \Pd, and 
$M$ a \tf\  $R$-module that is  the union of a countable ascending chain  
$ (2)$  of submodules such that, for every $n < \w$,

${1^\circ}.$    $M_n$ is a homogeneous \cd\ $R$-module of fixed type $\t$;

${2^\circ}.$   $M_n$ is an RD*-submodule of $M_{n+1}$.

\nin Then $M$ is \cd\ of type $\t$.  \end{theorem}  

\nin \begin{proof}  Condition (a) of Theorem \ref{fd} is satisfied by assumption ${1^\circ}.$ The stated necessary and sufficient condition (*) in this quoted theorem holds because of Theorem \ref{finite}, so our claim is immediate. 
\end{proof}
 
The following example will show that Theorem \ref{mainth} fails even for abelian groups if the condition of homogeneity is dropped.  We use the symbol $\vZ/p_1^\infty \dots p_k^\infty$ to denote the set of all rational numbers in whose denominators only powers of the primes $p_1, \dots, p_k$ occur.

\begin{example}  \rm Let $p_1, p_2, \dots, p_n, \dots$ be a list of distinct primes. Define
$$ A_0=\vZ, \  A_1 = \vZ/p_1^\infty \oplus \vZ/p_2^\infty , \ A_2 = \vZ/p_1^\infty p_3^\infty \oplus  \vZ/p_1^\infty p_4^\infty \oplus \vZ/p_2^\infty p_3^\infty \oplus  \vZ/p_2^\infty p_4^\infty, \dots$$ 
where from $A_{n-1}$ we pass to $A_n$ by replacing each summand by two copies of the direct sum of the summand  after adjoining to the denominators one of $p^\infty$ for the next two primes $p$ in the list.  In this way we get an ascending  chain $0 < A_1 < A_2 < \dots < A_n < \dots$ of \cd\ abelian groups if we use the diagonal embeddings (e.g. $A_1 \to A_2$ is induced by identifying $1 \in  \vZ/p_1^\infty $ with $(1,1) \in  \vZ/p_1^\infty p_3^\infty \oplus  \vZ/p_1^\infty p_4^\infty $ and $1 \in  \vZ/p_2^\infty $ with $(1,1) \in  \vZ/p_2^\infty p_3^\infty \oplus  \vZ/p_2^\infty p_4^\infty $).  Then each $A_n$ will be a pure subgroup in the following group in the chain. In order to justify our claim that the union $A = \cup_{n< \w} A_n$ is not \cd, assume by way of contradiction that $A$ is \cd\ and $J$ is a rank one summand of $A$.  Then $J$ is also a summand in the first link $A_m$ of the chain in which it is contained. The rank 1 summands of $A_m$ are fully invariant in $A_m$, so $J$ must be one of the summands in the given decomposition of $A_m$.  Manifestly,  $J$ has to be a summand in $A_{m+1}$ as well, but the construction shows that this is not the case. Thus $A$ cannot be \cd.   \end{example}
\smallskip

  Finally, we would like to apply our results to projective modules over integral domains $R$.
  
    We consider  the  case when the projective modules over $R$ are \fd. It is generally known that projective modules are direct sums of countably generated modules. Over a domain they are \fd\ if and only if they are direct sums of \fg\ modules. Rings over which the projective modules are direct sums of \fg\ modules are characterized by McGovern-Puninski-Rothberg \cite{MPR} for all associative rings. The integral domains for which this holds include all \Pds. 
  
  \begin{theorem} Assume that  projective modules over the integral domain $R$ are direct sums of \fg\ submodules. Then the union of a countable ascending chain $(2)$ of projective $R$-modules $M_n$ subject to condition {\rm (b)} is again projective if and only if condition  $(^*)$ of Theorem $\ref{fd}$ holds. \qedsymbol
  \end{theorem}

%=================================

\end{document}